\documentclass[reqno]{amsart}
\usepackage{amsmath,amsthm,amssymb,array}
\usepackage{hyperref}

\newtheorem{theorem}{Theorem}
\newtheorem{conjecture}[theorem]{Conjecture}
\newtheorem{corollary}[theorem]{Corollary}
\newtheorem{proposition}[theorem]{Proposition}
\newtheorem{lemma}[theorem]{Lemma}
\newtheorem{claim}{}[theorem]

\newcolumntype{O}{>{{}}c<{{}}}

\newcommand{\del}{ \backslash}
\newcommand{\bF}{\mathbb F}
\newcommand{\bZ}{\mathbb Z}
\newcommand{\cR}{\mathcal{R}}
\newcommand{\cT}{\mathcal{T}}

\newcommand{\BB}[3]{\mathrm{BB}(#1,#2,#3)}
\newcommand{\PG}[2]{\mathrm{PG}(#1,#2)}

\newcommand{\trianglehg}{\mathbb{T}}

\author{Kazuhiro Nomoto}
\author{Jorn van der Pol}
\address{Department of Combinatorics and Optimization, University of Waterloo, Waterloo, Canada}

\title{Tuza's conjecture for binary geometries}

\begin{document}

\begin{abstract}
	Tuza (\textit{A conjecture}, in Proceedings of the Colloquia Mathematica Societatis Janos Bolyai, 1981) conjectured that $\tau(G) \le 2\nu(G)$ for all graphs $G$, where $\tau(G)$ is the minimum size of an edge set whose removal makes $G$ triangle-free, and $\nu(G)$ is the maximum size of a collection of pairwise edge-disjoint triangles.
	Here, we generalise Tuza's conjecture to simple binary matroids that do not contain the Fano plane as a restriction. We prove that the geometric version of the conjecture holds for cographic matroids.
\end{abstract}

\maketitle

\section{Introduction}

Let $G$ be a simple graph. A \emph{(triangle) packing} is a set of pairwise edge-disjoint triangles in $G$, and a \emph{(triangle) hitting set} is a set of edges that meets every triangle in $G$. We write $\nu(G)$ for the size of a maximum packing, and $\tau(G)$ for the size of a minimum hitting set in $G$. It is easily seen that $\nu(G) \le \tau(G) \le 3\nu(G)$ for all graphs $G$. Tuza conjectured that the factor~3 can be improved to~2.

\begin{conjecture}[Tuza's conjecture~\cite{Tuza1981}]\label{conj:tuza}
	Let $G$ be a simple graph; then $\tau(G) \le 2\nu(G)$.
\end{conjecture}

If Tuza's conjecture is true, the constant~2 is best possible, as shown by the complete graphs~$K_4$ and~$K_5$.

Over the past 40 years, several special cases of Conjecture~\ref{conj:tuza} have been proven. It is now known that Conjecture~\ref{conj:tuza} holds for planar graphs~\cite{Tuza1990}, graphs without homeomorphic copy of~$K_{3,3}$~\cite{Krivelevich1995}, threshold graphs~\cite{Bonamy2021}, and a number of other graph classes. The full conjecture, however, remains wide open.

In this paper, we generalise Tuza's conjecture to the setting of simple binary matroids, that is, subsets of finite-dimensional binary projective spaces.

Unfortunately, Tuza's conjecture fails in general for simple binary matroids. The smallest counterexample is the Fano plane, $F_7 = \PG{2}{2}$, for which $\nu(F_7) = 1$ and $\tau(F_7) = 3$. A computer search among small binary matroids reveals that all simple binary matroids on at most~14 elements for which Tuza's conjecture fails contain a restriction isomorphic to the Fano plane. This inspires the following conjecture.

\begin{conjecture}\label{conj:geom}
	Let $M$ be a simple binary matroid that does not contain a restriction isomorphic to the Fano plane. Then $\tau(M) \le 2\nu(M)$.
\end{conjecture}

Every simple graph $G$ has an associated simple binary matroid, the cycle matroid $M(G)$, with the property that $\tau(M(G)) = \tau(G)$ and $\nu(M(G)) = \nu(G)$. Graphic matroids do not have restrictions isomorphic to the Fano plane, so the statement of Conjecture~\ref{conj:geom} implies that of Conjecture~\ref{conj:tuza}.

Haxell~\cite{Haxell1999} showed that $\tau(G) \le \tfrac{66}{23}\nu(G)$ for simple graphs $G$. One reason we believe it is natural to consider Tuza's conjecture in the geometric setting is that her proof only uses that $M(G)$ does not contain a Fano-restriction and therefore generalises, mutatis mutandis, to the geometric setting.

\begin{theorem}[\cite{Haxell1999}]
	Let $M$ be a simple binary matroid without Fano-restriction; then $\tau(M) \le \tfrac{66}{23}\nu(M)$.
\end{theorem}

In this paper, we prove that Conjecture~\ref{conj:geom} holds for cographic matroids, i.e.\ matroids whose duals are graphic.

\begin{theorem}\label{thm:cographic}
	If $M$ is a cographic matroid, then $\tau(M) \le 2\nu(M)$.
\end{theorem}

Theorem~\ref{thm:cographic} does not require the matroid $M$ to be simple. The notation $\tau(M)$ and $\nu(M)$ generalises to non-simple binary matroids in the obvious way (Tuza's conjecture for multigraphs was considered in~\cite{Chapuy2014}).

Whitney's planarity criterion~\cite{Whitney1932} asserts that a graph is planar if and only if its associated graphic matroid is cographic. Thus, Theorem~\ref{thm:cographic} implies that Tuza's conjecture holds for planar graphs, which was originally proved by Tuza.
\begin{corollary}[\cite{Tuza1990}]
	Let $G$ be a planar graph; then $\tau(G) \le 2\nu(G)$.
\end{corollary}

An even stronger generalization of Tuza's conjecture appears as Problem~1.8 in~\cite{AharoniZerbib2020}. Aharoni and Zerbib ask if, in a 3-uniform hypergraph without a tent-subgraph, the size of a minimum cover is at most twice the size of a maximum matching. (Here, a tent is the hypergraph on vertex set $\{1,2,\ldots,7\}$ with hyperedges $\{1,2,3\}$, $\{1,4,5\}$, $\{1,6,7\}$, and $\{3,5,7\}$.) The 3-uniform hypergraph on the elements of a simple binary matroid whose hyperedges are the triangles of the matroid is tent-free if and only if the matroid does not have a restriction isomorphic to the Fano-plane; thus a positive answer to the problem formulated by Aharoni and Zerbib implies a positive answer to Conjecture~\ref{conj:geom}. We thank Penny Haxell for pointing us to the paper~\cite{AharoniZerbib2020}.

\section{Preliminaries and notation}

\subsection{Simple binary matroids}

A simple binary matroid is a pair $M = (E,P)$, where $P = \PG{n-1}{2}$ is a finite-dimensional binary projective geometry of dimension $n-1$ and $E \subseteq P$. The rank of $M$ is~1 plus the dimension of the largest subgeometry of $P$ that contains $E$.

For matroids $M = (E_1,P_1)$ and $N=(E_2,P_2)$ we say that $M$ contains $N$ as a restriction if there exists a linear injection $\varphi\colon P_2 \rightarrow P_1$ such that $\varphi(E_2) \subseteq E_1$. We say that $M$ is $N$-free is $M$ does not contain $N$ as a restriction. We say that $M$ and $N$ are \emph{isomorphic} if there exists a linear bijection $\varphi\colon P_1 \rightarrow P_2$ such that $\varphi(E_1) = E_2$.

This definition of simple binary matroid is essentially the same as the standard definition of such matroids, except that our matroids are equipped with an extrinsic ambient space.

We abuse notation and write $\PG{n-1}{2}$ for both the $(n-1)$-dimensional binary projective geometry and the corresponding rank-$n$ simple binary matroid $(E,P)$ with $E = P = \PG{n-1}{2}$. We refer to the matroids $\PG{1}{2}$ and $\PG{2}{2} =: F_7$ as the \emph{triangle} and the \emph{Fano plane}, respectively.

\subsection{Graphic and cographic matroids}

Let $G$ be a graph, let $A$ be its vertex-edge incidence matrix, and write $n$ for the rank of~$A$. Let $A'$ be obtained from $A$ by restriction to a subset of its rows that is a basis of its row space $\cR(A)$.

The points of $P = \PG{n-1}{2}$ can be identified with the nonzero binary vectors in $\bF_2^n$. Let $E$ be the subset of $\PG{n-1}{2}$ formed by the columns of $A'$, then $(E, P)$ is a simple binary matroid. Although in our formalism $M(G)$ depends on the choice of $A'$, all such choices yields isomorphic matroids, and we will write $M(G)$ for the resulting corresponding matroid. More generally, a matroid is called \emph{graphic} if it can be obtained from a graph in this way.

In a similar fashion, if $A''$ is a matrix whose rows form a basis of the orthogonal complement $\cR(A)^\perp$, we can use the columns of $A''$ to define a matroid $M^*(G)$ (as in the case of $M(G)$, this matroid is unique up to isomorphism), and we call a matroid \emph{cographic} if it can be obtained from a graph in this way.

Triangles in $M(G)$ correspond to triangles in $G$, while triangles in $M^*(G)$ correspond to minimal edge cuts (bonds) of cardinality~3 in $G$ (which we will also call \emph{triads}).

\subsection{Packings and hitting sets}

Let $M = (E,P)$ be a simple binary matroid. A \emph{(triangle) packing} of $M$ is a collection of disjoint triangles contained in $M$; a \emph{(triangle) hitting set} of $M$ is a subset $X \subseteq E$ such that $(E\setminus X, P)$ is triangle-free. We write $\nu(M)$ for the maximum size of a triangle packing in $M$, and $\tau(M)$ for the minimum size of a triangle hitting set.

Alternatively, the parameters $\nu(M)$ and $\tau(M)$ can be formulated as the objective value of integer programmes. Write $\cT(M)$ for the collection of triangles of~$M$, then
\begin{equation*}
	\nu(M) = \max\left\{\sum_{T \in \cT(M)} x_T: \begin{array}{ll}\sum_{T \ni e} x_T \le 1 & \forall e \in E \\[1ex] \hphantom{\sum_{T \ni e}{}} x \in \bZ_{\ge 0}^{\cT(M)} &\end{array}\right\},
\end{equation*}
and
\begin{equation*}
	\tau(M) = \min\left\{\sum_{e \in E} y_e : \begin{array}{ll}\sum_{e \in T} y_e \ge 1 & \forall T \in \cT(M) \\[1ex] \hphantom{\sum_{e \in T}{}}y \in \bZ_{\ge 0}^E & \end{array}\right\}.
\end{equation*}

It is easily verified that $\nu(M(G)) = \nu(G)$ and $\nu(M(G)) = \tau(G)$ for any graph~$G$.

\subsection{Weighted binary matroids}

We generalise the notion of weighted graphs, as discussed by Chapuy et al.~\cite{Chapuy2014}, to matroids. Let $M = (E,P)$ be a simple binary matroid, and let $w\colon E \rightarrow \bZ_{\geq 0}$ be a weight function; we refer to the pair $(M,w)$ as a weighted binary matroid. The parameters $\nu(M)$ and $\tau(M)$ are easily generalised to the weighted setting:
\begin{equation*}
	\nu_w(M) = \max\left\{\sum_{T \in \cT(M)} x_T: \begin{array}{ll}\sum_{T \ni e} x_T \le w(e) & \forall e \in E \\[1ex] \hphantom{\sum_{T \ni e}{}}x \in \bZ_{\ge 0}^{\cT(M)} &\end{array}\right\},
\end{equation*}
and
\begin{equation*}
	\tau_w(M) = \min\left\{\sum_{e \in E} w_ey_e : \begin{array}{ll}\sum_{e \in T} y_e \ge 1 & \forall T \in \cT(M) \\[1ex] \hphantom{\sum_{e \in T}{}}y \in \bZ_{\ge 0}^E & \end{array}\right\}.
\end{equation*}

If $w_e = 1$ for all $e \in E$, then $\nu_w(M) = \nu(M)$ and $\tau_w(M) = \tau(M)$. The weighted versions of these parameters allow us to talk about binary matroids that may contain non-trivial parallel classes.
When $M = (E,P)$ and $w\colon E\rightarrow \bZ_{\geq 0}$ is a weight function, we can define a related weight function $w'\colon P \rightarrow \bZ_{\geq 0}$ by setting $w'(e) = w(e)$ if $e \in E$ and $w'(e) = 0$ otherwise. In that case, $\tau_{w'}(P) = \tau_w(M)$ and $\nu_{w'}(P) = \nu_w(M)$; thus, we may always assume that $E = P$.

\section{Cographic matroids}

In this section we prove Theorem~\ref{thm:cographic}.

\subsection{3-uniform hypergraphs}

Let $M = (E,P)$ be a simple binary matroid. It will be useful to encode the triangles of $M$ as a hypergraph $\trianglehg(M)$ on vertices $E$, in which a 3-set $T \in \binom{E}{3}$ forms a hyperedge if and only if $T$ is a triangle of $M$. The hypergraph $\trianglehg(M)$ is clearly 3-uniform, and as two triangles in $M$ intersect in at most one point, the hypergraph is linear as well. In terms of $\trianglehg(M)$, $\nu(M)$ is the size of a maximum matching in $\trianglehg(M)$, while $\tau(M)$ is the size of a minimum cover in $\trianglehg(M)$.

We will need the following standard result on hypergraphs.
\begin{lemma}\label{lemma:degree2_implies_linear_cycle}
Let $H=(V,E)$ be a $3$-uniform linear hypergraph in which the minimum degree is at least $2$. Then $H$ contains a linear cycle.
\end{lemma}

\begin{proof}
Let $P = \{e_1, \dots, e_t\}$ be a maximal linear path in $H$. Pick $v \in e_1 - e_2$. Since the degree of $v$ is at least $2$, there is an edge $e_0$ for which $e_0 \cap e_1 = \{v\}$. By maximality of $P$, there exists some $j \in \{2, \dots, t\}$ such that $|e_0 \cap e_j| \neq 0$. Take the smallest such $j$. Then $\{e_0,e_1,\dots, e_j\}$ forms a linear cycle.
\end{proof} 

A \emph{crown of size $k$} is a linear cycle on vertices $\{e_1,\dots,e_k, f_1,\dots,f_k\}$ with edges $\{e_i,f_i, e_{i+1}\}$, $i \in [k]$ (where we identify $k+1$ with $1$), with the additional property that these are the only edges in which the $e_i$ are contained. Crowns were introduced as an inductive tool in~\cite[Lemma 2]{Tuza1990}.

\subsection{Proof of Theorem~\ref{thm:cographic}}

We prove the following reformulation of Theorem~\ref{thm:cographic}. We remark that Claims~\ref{claim:at_least_two_triangles} and~\ref{claim:two_triangles} are just Properties~(a)--(c) in~\cite[Section~3]{Chapuy2014} adapted to our context, and follow from the same argument.

A note on language: In the proof of the following lemma, we consider both a simple cographic matroid $M$ and a graph $G$ such that $M = M^*(G)$. We use the terms ``triangle of $M$'', ``triad in $G$'', and ``hyperedge of $\trianglehg(M)$'' interchangeably, depending on the context.

\begin{lemma}\label{thm:cographic_weighted}
	Every weighted simple cographic matroid $(M,w)$ satisfies $\tau_w(M) \le 2\nu_w(M)$.
\end{lemma}

\begin{proof}
	Suppose that the lemma fails. Let $(M,w)$ be a counterexample for which $|E(M)| + w(E(M))$ is as small as possible. In the remainder of the proof, we write $E$ for $E(M)$. Let $G$ be a graph such that $M = M^*(G)$ for which $|V(G)|$ is as small as possible; as $M$ is simple, every edge-cut of $G$ has size at least~3. Small cases are easily checked, so we may assume that $|V(G)| \ge 3$.
	\begin{claim}\label{claim:at_least_two_triangles}
		For every $e \in E$: $w(e) \ge 1$ and $e$ is contained in at least two triangles of $M$.
	\end{claim}
	\begin{proof}[Proof of claim]
		If $w(e) = 0$ or $e$ is not contained in any triangle of $M$, then $(M',w')$ is a smaller counterexample, where $M' = M\del e$ and $w'$ is the restriction of $w$ to $E(M)-\{e\}$.
		So $w(e) \ge 1$. If $e$ is contained in exactly one triangle of $M$, say, $\{e,f,f'\}$, define a weight function $w'$ by setting $w'(x) = w(x) - 1$ for $x \in \{e,f,f'\}$ and $w'(x) = w(x)$ otherwise. Let $R$ be a minimal hitting set of $(M,w')$ with $w'(R) = \tau_{w'}(M)$; by minimality of $R$, $|R \cap \{e,f,f'\}| \le 2$, so $w(R) \le w'(R) + 2$. It follows that
		\begin{equation*}
			\tau_w(M) \le w(R) \le w'(R) + 2 = \tau_{w'}(M) + 2 \le 2\nu_{w'}(M) + 2 \le 2\nu_w(M),
		\end{equation*}
		which contradicts that $(M,w)$ is a counterexample.
	\end{proof}
	\begin{claim}\label{claim:two_triangles}
		If $e$ is in exactly two triangles of $M$, then $w(e) = 1$.
	\end{claim}
	\begin{proof}[Proof of claim]
		In view of the previous claim, it suffices to show that $w(e) \le 1$.
		Suppose, for the sake of contradiction, that $e$ is in exactly two triangles, and $w(e) \ge 2$. Call the two triangles $\{e,f,f'\}$ and $\{e,g,g'\}$. Consider the weighted simple cographic matroid $(M,w')$, where $w'(e) = w(e)-2$, $w'(x) = w(x)-1$ for all $x \in \{f,f',g,g'\}$, and $w'(x) = w(x)$ for all $x \in E(M)-\{e,f,f',g,g'\}$. Let $R$ be a minimal hitting set of $(M,w')$, so $w'(R) = \tau_{w'}(M)$. By minimality of $R$, if $R$ contains an element from $\{f,f'\}$ and an element from $\{g,g'\}$, then it does not contain $e$. It follows that $w(R) \le w'(R) + 4$, and hence that
		\begin{equation*}
			\tau_w(M) \le w(R) \le w'(R) + 4 = \tau_{w'}(M) + 4 \le 2\nu_{w'}(M) + 4 \le 2\nu_w(M),
		\end{equation*}
		which contradicts that $(M,w)$ is a counterexample.
	\end{proof}
	We now prove some basic properties of the graph~$G$.
	\begin{claim}\label{claim:conn}
		$G$ is 2-connected.
	\end{claim}
	\begin{proof}[Proof of claim]
		By minimality of $|V(G)|$, $G$ has no isolated vertices. If $G$ is not 2-connected, then $M$ is disconnected. It follows that $M$ has at least one component $M|X$, $X\subseteq E(M)$, such that the lemma already fails for $(M|X, w|X)$, which contradicts minimality of $(M,w)$.
	\end{proof}
	\begin{claim}
		$G$ is a simple graph.
	\end{claim}
	\begin{proof}[Proof of claim]
		Loops in $G$ are not contained in cuts, so we may assume that $G$ has no loops.
		Suppose, for the sake of contradiction, that $e$ and $e'$ are distinct parallel edges in $G$ and let $F$ be the maximal set of parallel edges containing both $e$ and $e'$. For each triad $T$ of $G$, either $T \cap F = \emptyset$, or $F \subseteq T$. It follows that if $|F| \ge 3$, then $e$ is contained in at most one triad of $G$, contradicting Claim~\ref{claim:at_least_two_triangles}. So we may assume that $F = \{e, e'\}$. By Claim~\ref{claim:at_least_two_triangles}, there are distinct elements $t$ and $t'$ such that $F\cup\{t\}$ and $F\cup\{t'\}$ are both triads of $G$. It follows that $\{t,t'\}$ contains a cut of~$G$, contradicting simplicity of~$M$.
	\end{proof}
	\begin{claim}\label{claim:crown}
		$\trianglehg(M)$ does not contain a crown.
	\end{claim}
	\begin{proof}[Proof of claim]
		Suppose, for the sake of contradiction, that $\trianglehg(M)$ contains a crown of size $k$, say $\{e_1, \ldots, e_k; f_1, \ldots, f_k\}$. By Claim~\ref{claim:two_triangles}, $w(e_i) = 1$ for all $i \in [k]$.
		
		If $k$ is even, say $k = 2q$, let $X = \{e_1, e_3, \dots, e_{2q-1}\} \cup \{f_1, f_3, \dots, f_{2q-1}\}$. Let $w'(x) = w(x) - 1$ for $x \in X$ and $w'(x) = w(x)$ otherwise. Let $R$ be a minimal hitting set such that $w'(R) = \tau_{w'}(M)$. Note that $w(R) = w'(R) + |R \cap X| \leq w'(R) + 2q$. Note also that $\nu_{w'}(M) + q \leq \nu_w(M)$: let $\cT'$ be a set of (not necessarily distinct) triangles of $M$ for which $\nu_{w'} (M) = |\cT'|$, then $\cT = \cT' \cup \{ \{e_i, f_i, e_{i+1}\} :  \text{$i$ is odd}   \}$ certifies that $\nu_{w} (M) \geq \nu_{w'}(M) + q$. It follows that
		\begin{equation*}
			\tau_w(M) \le w(R) \le w'(R) + 2q = \tau_{w'}(M) + 2q \le 2\nu_{w'}(M) + 2q \le 2\nu_w(M),
		\end{equation*}
		which contradicts that $(M,w)$ is a counterexample.
		
		If $k$ is odd, say $k = 2q + 1$, let $X = \{e_1, e_3, \dots, e_{2q + 1}\} \cup \{f_1, f_3, \dots, f_{2q-1}\}$; note that $|X| = 2q+1$. Let $w'(x) = w(x) - 1$ for $x \in X$ and $w'(x) = w(x)$ otherwise. Let $R$ be a minimal hitting set such that $w'(R) = \tau_{w'}(M)$. Note that $w(R) = w'(R) + |R \cap X|$. If $X \subseteq R$, then we may replace $R$ with $R \del \{e_1\}$ as $R \del \{e_1\}$ remains a hitting set. Hence we may assume that $|R \cap X| \leq 2q$, and therefore $w(R) \leq w'(R) + 2q$. Note also that $\nu_{w'}(M) + q \leq \nu_w(M)$; let $\cT'$ be a set of (not necessarily distinct) triangles of $M$ for which $\nu_{w'} (M) = |\cT'|$, then $\cT = \cT' \cup \{ \{e_i, f_i, e_{i+1}\} :  \text{$i$ is odd and $i \leq 2q-1$}   \}$ certifies that $\nu_{w} (M) \geq \nu_{w'}(M) + q$. It follows as before that
		\begin{equation*}
			\tau_w(M) \le w(R) \le w'(R) + 2q = \tau_{w'}(M) + 2t \le 2\nu_{w'}(M) + 2q \le 2\nu_w(M),
		\end{equation*}
		which contradicts that $(M,w)$ is a counterexample.
	\end{proof}
For $v \in V(G)$, write $\delta(v)$ for the set of edges incident with $v$. Call a triangle $T$ of $M$ a vertex-triangle if $T = \delta(v)$ for some $v \in V(G)$, and a non-vertex-triangle otherwise.
	\begin{claim}\label{claim:has_non_vertex_triangle}
		$M$ has a non-vertex-triangle.
	\end{claim}
	\begin{proof}[Proof of claim]
		By Claim~\ref{claim:at_least_two_triangles} and Lemma~\ref{lemma:degree2_implies_linear_cycle}, $\trianglehg(M)$ contains a linear cycle; let $C$ be such a cycle. If every triangle in $M$ is a vertex-triangle, then every element of~$M$ is in at most two triangles, and hence the maximum degree in $\trianglehg(M)$ is~2. This implies that $C$ is a crown, which contradicts Claim~\ref{claim:crown}.
	\end{proof}
Given a triangle $T$ of $M$, denote by $G_1(T)$ and $G_2(T)$ the two connected components of $G\del T$; we may assume that $|E(G_1(T))| \le |E(G_2(T))|$. 	Among all non-vertex-triangles, let $T$ be one for which $|E(G_1(T))|$ is as small as possible, and write $X = E(G_1(T))$. As $T$ is not a vertex-triangle, the set $X$ is non-empty. 
	\begin{claim}\label{claim:triangle-left-right}
		Every triangle of $M$ is contained in $E(G_1(T)) \cup T$ or in $E(G_2(T)) \cup T$.
	\end{claim}
	\begin{proof}[Proof of claim]
		The claim clearly holds for $T$. Let $S \neq T$ be a triangle of $M$. Since $M$ is simple and binary, it follows that $|S \cap T| \le 1$ and $r_M(S \cup T) = 4 - |S\cap T|$. By Claim~\ref{claim:conn}, $M$ is connected and hence~$r(M) = |E|-|V(G)|+1$. It follows that
		\begin{equation*}
			r_{M(G)}(E \setminus (S\cup T))
			= r_M(S\cup T) + |E\setminus(S\cup T)| - r(M)
			= |V(G)|-3,
		\end{equation*}
		so the graph $G\setminus(S\cup T)$ has three connected components; thus there exists $i \in \{1,2\}$ such that $G_i(T)\setminus S$ is connected, while $G_{3-i}(T)\setminus S$ is not.
		
		Suppose, for the sake of contradiction, that the claim fails for $S$. As $|S \cap T| \le 1$, clearly $1 \le |E(G_j(T))\cap S| \le 2$ for $j \in \{1,2\}$.
		
		If $|E(G_i(T))\cap S| = 1$, then, since $G_i(T)\setminus S$ is connected, the unique element in $E(G_i(T)) \cap S$ is contained in a cycle of $G_i(T)$; in this case, $G$ has a cut and a cycle that intersect in a single element: a contradiction, so $|E(G_i(T))\cap S| = 2$ and consequently $|E(G_{3-i}(T))\cap S| = 1$.
		
		Let $s$ be the unique element in $E(G_{3-i}(T))\cap S$. As
		\begin{equation*}
			\begin{split}
				|V(G)|-3
					&= r_{M(G)}(E \setminus (T\cup\{s\})) \\
					&= r_M(T \cup \{s\}) + |E\setminus(T \cup \{s\})| - r(M) \\
					&= r_M(T \cup \{s\}) + |V(G)| - 5,
			\end{split}
		\end{equation*}
		we must have that $r_M(T \cup \{s\}) = 2$, which contradicts that $M$ is simple.
	\end{proof}
	\begin{claim}\label{claim:min_degree}
		The minimum degree in~$G_1(T)$ is at least~2.
	\end{claim}
	\begin{proof}[Proof of claim]
		Suppose, for the sake of contradiction, that $G_1(T)$ contains a vertex~$v$ of degree at most~1.
		Let $(\alpha, \beta) = (d_{G_1(T)}(v), d_G(v))$. Clearly, $\alpha \le \beta \le \alpha + 3$.
		
		If $\beta = 0$, then $G$ has an isolated vertex, contradicting Claim~\ref{claim:conn}.
		
		If $\beta \in \{1,2\}$, then $G$ has a vertex of degree~$\beta$, contradicting simplicity of~$M$.
		
		If $(\alpha,\beta) = (0,3)$, then $T$ is a vertex-triangle of $M$: a contradiction.
		
		If $(\alpha,\beta) = (1,3)$, then $T \triangle \delta_G(v)$ is a cut of size~2 in~$G$: a contradiction.
		
		If $(\alpha,\beta) = (1,4)$, then $\delta_G(v)\setminus T$ is a cut of size~1 in~$G$: a contradiction.
		
		As this list exhausts all possible pairs $(\alpha, \beta)$, it follows that the minimum degree in $G_1(T)$ is at least~2.
	\end{proof}
	By Claim~\ref{claim:min_degree}, the graph $G_1(T)$ contains at least one cycle; among all such cycles, let $C = v_1v_2\ldots v_pv_1$ be one that is shortest.
	\begin{claim}
		$\delta_G(v_i)$ is a vertex-triangle for all $i \in [p]$; if $e$ is incident with $v_i$ and $v_j$, then the only triangles of $M$ containing $e$ are $\delta_G(v_i)$ and $\delta_G(v_j)$.
	\end{claim}
	\begin{proof}[Proof of claim]
		Let $e \in \delta_G(v_i)\setminus T$. By Claim~\ref{claim:at_least_two_triangles}, $e$ is contained in at least two triangles of~$M$; let $T'$ be a triangle of $M$ containing $e$. By Claim~\ref{claim:triangle-left-right}, $T' \subseteq X \cup T$. As $T$ is the unique non-vertex-triangle in $X\cup T$, $T'$ must be a vertex-triangle. We conclude that $e$ is in two vertex-triangles, one of which must be $\delta_G(v_i)$.
	\end{proof}
	Let $K$ be the set of edges along the cycle $C$, and let $F = \bigcup_{i=1}^p \delta_G(v_i) \setminus K$. As the cycle $C$ is of minimum length, $K \cup F$ is a crown of size $p$ in $\trianglehg(M)$ with hyperedges $\{\delta_G(v_i): i \in [p]\}$, in which the elements of $K$ have degree~2. This contradicts Claim~\ref{claim:crown}.
\end{proof}

\section{Geometries}

We conclude this paper by proving the geometric version of Tuza's conjecture in a few special cases.

\subsection{Projective geometries}

Consider $\PG{n-1}{2}$. When $n=1$ or $n=2$, Tuza's conjecture holds trivially, but when $n=3$ it fails. We will assume that $n\ge 4$.

The removal of a hyperplane makes $\PG{n-1}{2}$ triangle-free, and no smaller set has the same property. It follows that
\begin{equation}\label{eq:tau_projective}
	\tau(\PG{n-1}{2}) = 2^{n-1}-1.
\end{equation}

A \emph{spread} is a partition of a projective geometry into lower-dimensional subgeometries. The following result, phrased here in matroidal terms, can be found in~\cite[p.\ 29]{Dembowski1968}.
\begin{theorem}\label{thm:spread}
	The binary projective geometry $\PG{n-1}{2}$ can be partitioned into subgeometries isomorphic to $\PG{d-1}{2}$ if and only if $d|n$.
\end{theorem}

An immediate consequence of this result is that binary projective geometries of even rank can be partitioned into triangles, and hence
\begin{equation}\label{eq:nu_projective_even}
	\text{$n$ even} \Longrightarrow \nu(\PG{n-1}{2}) = \tfrac{1}{3}(2^n-1).
\end{equation}

Binary projective geometries of odd rank cannot be partitioned into triangles; however \emph{partial} spreads were studied by Beutelspacher~\cite[Theorems~4.1--4.2]{Beutelspacher1975}, who showed that binary projective geometries of odd rank can be partitioned into triangles and four additional points (in fact, a 4-circuit, but that is not important here), and hence
\begin{equation}\label{eq:nu_projective_odd}
	\text{$n$ odd} \Longrightarrow \nu(\PG{n-1}{2}) = \tfrac{1}{3}(2^n - 5).
\end{equation}

Combining \eqref{eq:tau_projective}--\eqref{eq:nu_projective_odd}, Tuza's conjecture for projective geometries of rank at least~4 follows.
\begin{proposition}
	If $n \ge 4$, then $\tau(\PG{n-1}{2}) \le 2 \nu(\PG{n-1}{2})$. Moreover, $\tau(\PG{n-1}{2})/\nu(\PG{n-1}{2}) \to \tfrac{3}{2}$ as $n \to \infty$.
\end{proposition}

\subsection{Bose--Burton geometries}

For $n \ge k$, let $P = \PG{n-1}{2}$ and let $Q$ be a subgeometry of $P$ of rank $n-k$. We refer to the matroid $(P-Q,P)$ as the Bose--Burton geometry $\BB{n}{k}{2}$. It is the maximum-size simple binary matroid of rank $n$ that does not contain $\PG{k-1}{2}$ as a submatroid~\cite{BoseBurton1966}.

It is easily checked that $\BB{n}{0}{2}$ is empty and $\BB{n}{1}{2}$ is triangle-free. In both cases, Tuza's conjecture holds trivially. For the remainder of this section, we will therefore assume that $k \ge 2$.

\begin{lemma}\label{lemma:tau_boseburton}
	Let $n \ge k \ge 2$. Then $\tau(\BB{n}{k}{2}) = 2^{n-1} - 2^{n-k}$.
\end{lemma}
\begin{proof}
	With $P$ and $Q$ as above, let $H$ be a hyperplane of $P$ with $Q \subseteq H$. The matroid $M' = ((P\setminus Q)\setminus(H\setminus Q), P)$ is triangle-free, so $\tau(\BB{n}{k}{2}) \le |H\setminus Q| = 2^{n-1} - 2^{n-k}$. As any triangle-free matroid of rank at most $n$ has at most $2^{n-1}$ points, the upper bound is in fact an equality.
\end{proof}

We next show that $\BB{n}{2}{2}$ can be partitioned into triangles, using a variant of the proof of~\cite[Theorem~4.2]{Beutelspacher1975}.

\begin{lemma}\label{lemma:nu_boseburton_k2}
	Let $n \ge 2$. $\BB{n}{2}{2}$ can be partitioned into triangles. In particular, $\nu(\BB{n}{2}{2}) = 2^{n-2}$.
\end{lemma}
\begin{proof}
	The claim is trivial for $n=2$, so we may assume that $n \ge 3$. As above, $\BB{n}{2}{2} = (P-Q, P)$. Consider an embedding of $P$ in $\PG{2n-5}{2}$. By Theorem~\ref{thm:spread}, $\PG{2n-5}{2}$ can be partitioned into subgeometries of rank $n-2$. Let $\mathcal{U}$ be such a partition. By symmetry, we may assume that $Q \in \mathcal{U}$.
	
	We claim that
	\begin{equation*}
		\left\{U \cap P : U \in \mathcal{U}\setminus\{Q\}\right\}
	\end{equation*}
	is the required partition of $P\setminus Q$ into triangles.
	
	For each $U \in \mathcal{U}\setminus\{Q\}$, we find
	\begin{equation*}
		\begin{split}
			r(U \cap P)
				&= r(U) + r(P) - r(\langle U, P\rangle) \\
				&\ge r(U) + r(P) - (2n-4) \\
				&= (n-2) + n - (2n-4) = 2,
		\end{split}
	\end{equation*}
	so each such $U$ intersects $P$ in (at least) a triangle. As $|\mathcal{U}\setminus\{Q\}| = 2^{n-2}$, it follows that $U \cap P$ is a triangle for each $U$. This proves the claim.
\end{proof}

\begin{lemma}\label{lemma:nu_boseburton}
	$\nu(\BB{n}{k}{2}) \ge \frac{1-4^{-\lfloor k/2\rfloor}}{3} 2^n$ for all $n \ge k \ge 2$.
\end{lemma}
\begin{proof}
	Write $M = \BB{n}{k}{2}$. Let $P = \PG{n-1}{2}$ and let $Q$ and $R$ be subgeometries of rank $n-2$ and $n-k$, respectively, such that $P \supset Q \supset R$. Then $M = (P\setminus R,P)$, $(P\setminus Q, P) = \BB{n}{2}{2}$, and $(Q\setminus R, Q) = \BB{n-2}{k-2}{2}$. As $P\setminus R$ is the disjoint union of $P\setminus Q$ and $Q\setminus R$, it follows that
	\begin{equation*}
		\setlength{\arraycolsep}{0pt}
		\begin{array}{r O c O c O c O c}
			\nu(\BB{n}{k}{2})
				&\ge& \nu(\BB{n}{2}{2}) &+& \nu(\BB{n-2}{k-2}{2}) \\
				&=& 2^{n-2} &+& \nu(\BB{n-2}{k-2}{2}).
		\end{array}
	\end{equation*}
	The claim now follows by induction and $\nu(\BB{n}{0}{2}) = \nu(\BB{n}{1}{2}) = 0$.
\end{proof}

Lemma~\ref{lemma:tau_boseburton} and Lemma~\ref{lemma:nu_boseburton} now readily imply Tuza's conjecture for Bose--Burton geometries.
\begin{proposition}
	$\tau(\BB{n}{k}{2}) \le 2\nu(\BB{n}{k}{2})$ for all $n \ge k \ge 2$.
\end{proposition}

\subsection{Matroids with critical number at most 2}

The critical number $\chi(M)$ of an $n$-dimensional simple binary matroid $M = (E,P)$ is the smallest integer $k \ge 0$ such that $M$ is a restriction of the Bose--Burton geometry $\BB{n}{k}{2}$; alternatively, it is the smallest integer $k \ge 0$ such that $P\setminus E$ contains a subgeometry of $P$ of dimension $n-k$. In particular, if $\chi(M) \le 2$, then there is a rank-$(n-2)$ subgeometry of $P$ such that $E \subseteq P\setminus Q$. Matroids with critical number at most~2 do not have restrictions isomorphic to the Fano plane.

\begin{proposition}
	Let $M = (E,P)$ be a simple binary matroid of rank $n$ such that $\chi(M) \le 2$. Then $\tau(M) \le 2\nu(M)$.
\end{proposition}
\begin{proof}
	As $\chi(M) \le 2$, $P$ has a rank-$(n-2)$ subgeometry $Q$ such that $E \subseteq P\setminus Q$. Moreover, $P$ has exactly three hyperplanes that contain $Q$ as a subgeometry. As $E$ is disjoint from $Q$, every triangle of $M$ intersects each of these hyperplanes. Consequently, $E$ can be coloured by three colours such that each triangle receives all colours. The claim now follows from the 3-uniform version of Ryser's conjecure, which was proved by Aharoni~\cite{Aharoni2001}.
\end{proof}

\bibliographystyle{alpha}
\bibliography{tuza}
\end{document}